\documentclass[a4paper,11pt]{article}
\usepackage{amsmath}
\usepackage{amssymb}

\renewcommand{\P}{\mbox{\bf{P}}}

\newcommand{\defeq}{\stackrel{\rm def}{=}}

\newcommand{\prob}{\mbox{\bf{Pr}}}
\newcommand{\expect}{\mbox{\bf{Ex}}}

\newtheorem{theorem}{Theorem}
\newtheorem{fact}[theorem]{Fact}
\newtheorem{proposition}[theorem]{Proposition}
\newtheorem{corollary}[theorem]{Corollary}
\newtheorem{claim}[theorem]{Claim}
\newtheorem{lemma}[theorem]{Lemma}
\newtheorem{definition}[theorem]{Definition}

\newcommand{\bproof}{\noindent{\it Proof}}
\newcommand{\eproof}{\hspace*{\fill}$\Box$~~~~~\bigskip}
\newenvironment{proof}{\bproof. }{\eproof}

\newenvironment{psketch}{\bproof~{\it Sketch}. }{\eproof}

\title{Various proofs of the Fundamental Theorem of Markov Chains}

\author{
	Somenath Biswas\footnote{E-mail: {\tt sb@iitgoa.ac.in}}\\
School of Mathematics and Computer Science\\
IIT Goa
}
\date{}
\begin{document}
\maketitle
\begin{abstract}
This paper is a survey of various proofs of the so called {\em fundamental theorem of Markov chains}: every ergodic Markov chain has
a unique positive stationary distribution and the chain attains this distribution in the limit independent of the initial distribution the 
chain started with. As Markov chains are stochastic processes, it is natural to use probability based arguments for proofs. 
At the same time, the dynamics of a Markov chain is completely captured by its initial distribution, which is a vector, and
its transition probability matrix. Therefore, arguments based on matrix analysis and linear algebra  can also be used. 
The proofs discussed below use one or the other of these two types of arguments, except in one case where the argument is graph theoretic.\\ 
Appropriate credits to the various proofs are given in the main text.

Our first proof is entirely elementary, and yet the proof is also quite simple. The proof also suggests a mixing time bound, which
we prove, but this bound in many cases will not be the best bound. One approach in proving the fundamental theorem breaks the proof 
in two parts:\\
(i) show the existence of a unique positive stationary distribution for irreducible Markov chains, and\\
(ii) assuming that an ergodic chain does have a stationary distribution, show that the chain will converge in the limit to that 
distribution irrespective of the initial distribution.\\ 
For (i), we survey two proofs, one uses probability arguments, and the other
uses graph theoretic arguments. For (ii), first we give a coupling based proof (coupling is a probability based technique), the other 
uses matrix analysis. Finally, we give a proof of the fundamental theorem using only linear algebra concepts.
\end{abstract}
\newpage
\tableofcontents

\section{Introduction}
Every ergodic, that is, both irreducible and aperiodic, finite Markov chain has a unique positive stationary distribution and 
this distribution is attained by the chain
in the limit, starting with any initial probability distribution. This fact,
known as the fundamental theorem of Markov chains, is mainly responsible for the wide variety of applications that
Markov chains find in  diverse fields. Let us briefly mention one application: the page rank algorithm. 
Without doubt, the phenomenal success of Google started with the discovery of this algorithm. When a user makes a search in the 
Internet, the underlying database identifies a set of web pages relevant to the search query. The number of such
pages usually will be very large, possibly in hundreds of thousands. The challenge is to present to the user a certain number of
pages in the decreasing order of their relevance. The page rank algorithm solves this challenge by
viewing the set of discovered pages
as a directed graph, where the edges are the hyperlinks, going from  web pages to web pages. The algorithm considers 
this digraph  to be defining a random walk Markov
chain, the state space of which is the set of web pages discovered by the database relevant to the search query, 
and the transition probability of going from the page $p_i$ to the page $p_j$ is $k/d$,
where $d$ is the total number of hyperlinks occurring in the page $p_i$, out of which $k$ are to the page $p_j$. By adding some extra 
hyperlinks, if necessary, the page rank algorithm first makes the random walk Markov chain ergodic, and then for some $n$, identifies 
the $n$ webpages which have the $n$ highest probabilities in an appropriate approximation 
of the stationary distribution of the
random walk Markov chain. These web pages are displayed with decreasing order of their stationary probabilities. Thus, the ranking 
algorithm tries to capture the intuitive notion of ordering of relevance of web pages by their stationary probabilities in a random walk,
which appears reasonable. That there  exists a unique stationary distribution of the random walk Markov chain and we can come 
closer and closer to the stationary distribution, by doing a random walk independent of where we start the walk, is a consequence of the 
fundamental theorem which the page rank algorithm makes use of.

Like most fundamental results, this theorem too can be proved 
in various ways, this paper is a survey of a number of proofs of the fundamental theorem.
%\footnote{A shorter version of
%this paper was originally prepared for the students of a course on Markov chains and their applications that I had taught
%recently. Since many students of that class had only elementary knowledge of linear algebra and probability, the material
%was presented with such readers in mind. In this version too, very little background in these two subjects have been assumed.} 
It is natural to use reasoning based on probability for proving properties of Markov chains, as they are stochastic processes. 
At the same time, the dynamics of Markov chains of concern here are entirely captured by their initial distributions,
which are vectors,  and their transition probability matrices. Therefore,  matrix analysis and linear algebra 
can also be used for dealing with Markov chains. Proofs discussed here use one or the other of these two 
types of reasoning, except in one case, in which the heart of the argument is, surprisingly, graph theoretic.   

The next section deals with the basic definitions and notations. Section 3 details a very simple and an elementary
proof of the fundamental theorem. Some proofs of the theorem splits the proof in two parts, they first establish that
a stationary distribution, in which each component is positive, exists for irreducible. The second part proves, assuming a stationary 
distribution does exist, that every initial distribution converges to that  stationary distribution, thereby proving
simultaneously, the uniqueness of the stationary probability. Section 4 deals
with this approach. Section 5 provides a proof entirely based on linear algebra.  
%%%%%The second proof that we present is through probabilistic arguments, while the third proof is entirely through
%%%%%linear algebra. 
\section{Preliminaries}
The notations used here and some basic definitions are as follows, for greater details, we refer to \cite{haggstrom}, 
\cite{norris}, \cite{mixing-book}. 
The Markov chains
considered here have finite state spaces, are discrete time, and time-homogeneous. The symbol $\Omega$ has been generally used to
denote a state space. A Markov chain with a finite state space is called a {\em finite} Markov chain. A sequence $(X_i)_{i\geq 0}$ 
of random variables taking their values from a finite set $\Omega$ is a Markov
chain over the state space $\Omega$ if they satisify the property 
\[\Pr(X_{n+1} = i_{n+1}|X_0 = i_0, X_1 = i_1,\ldots,  X_n = i_n) = \Pr(X_{n+1} = i_{n+1}|X_n = i_n)\] for all $i_0,\ldots, i_{n+1}\in\Omega, n\geq 0$.
As we deal with only time homogeneous chains, $(X_i)_{i\geq 0}$ further satisfies: for all $i,j\in\Omega$, and for all $n\geq 0$,  
$\Pr(X_{n+1} = j|X_n = i)$ is independent of $n$, and therefore,
there is an $\Omega\times\Omega$ matrix $P$, called the {\em transition probability matrix}, or simply the {\em transition matrix}, 
such that  $\Pr(X_{n+1} = j|X_n = i = P(i,j)$, for all $i, j\in\Omega$ and for all $n\geq 0$. 
For any matrix $A$, we denote its $ij$th entry as $A(i,j)$ or as $a_{ij}$, or, as well as $A_{ij}$. 
%Let $P$ denote the transition probality matrix
%of a chain under consideration, $P_{ij}$ denoting the probability of the chain moving to state $j$ at time $t+1$,
%given that it is in state $i$ at time $t$.\footnote{We will also use $P(x,y)$ to denote the
%transition probability of moving from state $x$ to state $y$. 
We use the notation $A(x,\cdot)$ for the $x$th row of $A$. Clearly, for a transition matrix $P$, $P(x, \cdot)$ is the
next state probability distribution when the present state is $x$.
It is easy to see (say, by induction) that $P^k$ is the  
$k$ step transition probability matrix, that is, its $ij$th entry, $P^k_{ij}$, is the
probability of the chain moving to state $j$ at time $t+k$, given that it
is in state $i$ at time $t$.  A probability distribution $\pi$ on the state space is a {\em stationary distribution}
of the chain if $\pi P = \pi$. A matrix such as a transition matrix $P$, which has every entry $P_{ij}\in [0,1]$, and
with each row entries
summing to $1$, is called a {\em stochastic} matrix. If every entry of a matrix $A$ is positive, i.e.,$A_{ij} > 0$ for all $i,j$, 
then $A$ is called
{\em positive}, often denoted as $A > 0$. 
%A {\em finite} Markov chain is one which has a finite state space. 

Let $G(P)$ denote the underlying transition graph: it is a directed graph that has a vertex for every state
and has an edge from the vertex labeled $i$, representing the state $i$, to the vertex labeled  $j$, for all $i, j\in\Omega$ 
iff $P_{ij}>0$. If $G(P)$ is strongly connected then the corresponding Markov
chain is called {\em irreducible}.  The term {\em aperiodicity} means that that for every state $s$, the gcd of the lengths
of all walks in $G(P)$, each of which starts at vertex $s$ and ends back in $s$ is $1$.  A Markov chain is {\em ergodic} if it
is both irreducible and aperiodic. 

The fundamental theorem
of Markov chains states 
\begin{theorem}[Fundamental Theorem of Markov chains] 
If a discrete time, finite, and time-homogeneous Markov chain is ergodic then  
it will have  a unique stationary distribution that assignns positive probability to every state and the chain, 
starting with any initial distribution, will attain 
the stationary distribution in the limit.\footnote{The result extends  also for
chains with denumerable state spaces.}
\end{theorem}

It is easy to see that both irreducibility and aperiodicity are necessary 
for the fundamental theorem to hold: if the chain has, say, two connected components, and if the initial starting state
is one of these components, then all future states will be from the same component. Therefore, the limiting distribution
cannot be independent of the initial distribution on   states. On the other hand, if the transition graph is, say, bipartite,
violating the aperiodicity condition, then again there will not exist a stationary distribution which the chain reaches 
in the limit because, 
if the initial state at $t=0$ is in one of the components, then
the state at every even $t$ will be from this component, and for every odd $t$ the state will be from the other component.
It is truly remarkable that these two easy-to-see necessary conditions are also sufficient conditions for the fundamental
theorem to hold, which is a strong result also in the following sense: whereas the existence of a unique stationary distribution for a 
stochastic matrix that is irreducible does follow from Perron-Frobenius Theorem, the stationary distribution may not ever be 
attained,\footnote{Consider
the two state chain with the state space $\{0,1\}$, in each move, if the chain is currently in state $i$, it moves with probability
$1$ to the other state, namely the state $1-i$. This chain has the unique stationary probability distribution $[0.5 0.5]$, but the 
stationary distribution is never attained.} the fundamental theorem guarantees that for ergodic chains not only a unique stationary 
distribution exists, but also convergence to it, {\em no matter what might be the initial distribution}. 

In some of the proofs we shall need the following:
\begin{claim}
\label{positive-P}
If a finite state space, discrete time, and time homogeneous Markov chain with transition probability matrix $P$ 
is ergodic, that is, both aperiodic and irreducible, then there exists a finite $k$ such that for all $l \geq k$,
$P^l$ is positive, that is, for all $i,j$, $P^l_{ij} > 0$.
\end{claim}.
A proof of the Claim above is given in \cite{haggstrom}, we briefly sketch the proof idea in a somewhat different way than 
done in \cite{haggstrom}. Let $M = (X_i)_{i\geq 0}$, with state
space $\Omega$ be an ergodic 
Markov chain, and let $P$ be its transition matrix. We recall that $G(P)$ is used to denote the underlying transition graph of the 
chain. We shall use the following fact repeatedly in the proof sketch below:
\begin{fact}
\label{path-prob}
For every two vertices $u,v$ in $G(P)$, and for every walk $w$ from $u$ to $v$ of length $|w|\geq 1$ in $G(P)$, and for every $n\geq 0$,
by definition of $G(P)$, $\Pr(X_{n+|w|} = v|X_n = u) > 0$.
\end{fact} 
Next,
for any state $i\in\Omega$, let $A_{ii}$, a set of positive integers, be defined as follows:
\[A_{ii} = \{|w| |w \mbox{ is a walk from } i\mbox{ to } i,\ |w|\geq 1\}\]
We note that since the chain is aperiodic, by definition, the gcd of the numbers in $A_{ii}$ is $1$. Further, the set of numbers in
$A_{ii}$ is closed under addition: if $G(P)$ has two walks $w_1,\ w_2$, each from vertex $i$ back to $i$, which means that
$|w_1,\ |w_2|\in A_{ii}$;
we can concatenate the two walks to get a longer walk from $i$ back to $i$, and therefore $|w_1|+|w_2|$ will also be in $A_{ii}$.

We can now use the following number theoretic fact (Lemma 4.1 of \cite{haggstrom}):
\begin{fact}[\cite{bre}]
Let $A = \{a_1, a_2,\ldots\}$ be a set of positive integers which is
\begin{enumerate}
\item nonlattice, meaning the gcd$\{a_1,a_2,\ldots\}=1$, and 
\item closed under addition, meaning that if $a\in A$ and $a'\in A$ then $a+a'\in A$
\end{enumerate}
then there exists an integer $N<\infty$ such that $n\in A$ for all $n\geq N$.
\end{fact}
Clearly, the above Fact applies to the set $A_{ii}$, and let $N_i$ denote the number such that for all $n\geq N_i$, 
there will be a walk of length $n$ from vertex $i$ back to $i$, which implies by Fact \ref{path-prob} that with a positive
probability the Markov chain $M$ can move from state $i$ back to state $i$ in $n$ steps. 

Next, let $j$ be an arbitrary state of the Chain $M$. As the Chain $M$ is assumed to be irreducible, there 
exists at least one path from vertex 
$i$ to $j$ in the transition graph $G(P)$, we consider one such path which is of length, say, $l_{ij}$.\\ 
Let $K_{ij}\defeq N_i + l_{ij}$. It is clear that for any $m\geq K_{ij}$, there is a walk in $G(P)$ of length $m$ from 
$i$ to $j$: first use the first $m - l_{ij}$ steps to go from $i$ back to $i$, and then use the last $l_{ij}$ steps to move
from $i$ to $j$. It follows from Fact \ref{path-prob} that for every $m\geq K_{ij}$, the $(i, j)$th entry of $P^m_{ij} > 0$. Let $K$ denote\\
$K\defeq\max\{K_{ij}|i,j\in\Omega\}$. We have therefore, that for every $l\geq K$, the matrix $P^l$ is positive. This completes our proof
sketch of the above Claim \ref{positive-P}.

Therefore, for the purpose of proving the fundamental theorem, we lose no generality by assuming the
transition probability matrix to be positive, an assumption we often make here in our proofs.

\section{A simple and elementary proof}
%If the Markov chain with transition matrix $P$ is both irreducible and aperiodic then it is easy to see that
%there exists some finite $k$ such that for any two states $i, j$, the probability of reaching $j$ starting from
%$i$ will be non-zero if the chain takes $k$ or more steps. For what is to be proved, we can therefore assume
%without any loss of generality, that every entry $P_{ij}$ of the transition matrix of our Markov chain is
%non-zero.   
The first proof that we discuss is an elementary and a very simple proof, which is 
due to Alessandro Panconesi \cite{panconesi}, who in turn credits the proof idea to David Gilat.
Interestingly, the same proof idea was used by Markov himself in his original 1906 paper on Markov chains,
see Theorem 4.1 of the survey on Markov's life and his work \cite{basharin}.
%Some intermediate steps omitted by Panconesi have been included here for greater clarity. 
The proof also implies a bound which we provide on how quickly
an ergodic Markov chain comes close to its stationary distribution; we also discuss how good the bound is. 
\subsection{The proof}
Let $P$ be the transition matrix of an ergodic Markov chain, we assume that $P$ is 
positive.
%As the chain
%is both aperiodic and irreducible, we can assume w.l.o.g that $P$ is positive, that is, $P_{ij}>0$ 
%for all $i,j$. 
We define $P^\infty$ as 
\[P^\infty\defeq\lim_{n\rightarrow\infty}P^n\]
We prove that
\begin{proposition}
\label{matrix-P-infinity}
For any $n\times n$ positive stochastic matrix $P$, there exist $\pi_1, \pi_2,\ldots, \pi_n$, with each $i$, $0<\pi_i<1$
and $\sum_{i=1}^n\pi_i = 1$ such that
\[P^\infty = \left[\begin{array}{cccc}
      \pi_1  &\pi_2 &\ldots  &\pi_n\\ 
      \pi_1  &\pi_2 &\ldots  &\pi_n\\ 
        .    &.     &\ldots  &. \\
      \pi_1  &\pi_2 &\ldots  &\pi_n 
     \end{array}\right]\]
Further, $\pi\ \defeq\ (\pi_1,\pi_2,\ldots,\pi_n)$ is a stationary distribution of $P$, and for any
initial distribution $\sigma = (\sigma_1,\sigma_2,\ldots,\sigma_n)$, the sequence $(\sigma^{(i)})_{i\geq 1}$,
defined as $\sigma^{(1)}\defeq\sigma$ and $\sigma^{(i+1)}\defeq\sigma^{(i)}P$, reaches $\pi$ as $i$
goes to infinity.  
\end{proposition}

\begin{proof}
That, for any column of $P^\infty$, every element in the column has the same value follows from the
following Claim: 
\begin{claim}
\label{P-infinity}
Let $P$ be a stochastic matrix with each entry positive. For an arbitrary fixed column, say $k$, of $P$,
let $m^{(i)}$ and $M^{(i)}$ denote respectively the minimal and the maximal entry of the column $k$ of $P^i$,
the $i$th power of $P$. Then,\\
(a) the sequence $(m^{(i)})_{i\geq 1}$ is non-decreasing,\\ 
(b) the sequence $(M^{(i)})_{i\geq 1}$ is non-increasing, and,\\
(c) $\Delta^{(i)}\defeq M^{(i)}-m^{(i)}$ goes to $0$ as $i$ goes to infinity.
\end{claim}
The following Corollary is immediate  from the Claim above. 
\begin{corollary}
\label{pi-k}
If $P$ is a positive, stochastic matrix then there exists a value, say $\pi_k$, such that each element
of the $k$th column of $P^{(i)}$ approaches $\pi_k$ as $i$ goes to infinity.
\end{corollary}
%the same value, say, $\pi_k$, where $P^{(\infty)}$ is $\lim_{i\rightarrow\infty}P^{(i)}$. 
The corollary above follows from Proposition \ref{P-infinity}
because as $P$ is a positive, stochastic matrix, each $P^i$ too is positive and stochastic for all $i>0$. 
Therefore, $0 < m^{(i)}$ and $M^{(i)} < 1$, for each $i\geq 1$. As every bounded non-decreasing (or, non-increasing) sequence has
a limit, the sequences in (a) and (b) of the Proposition \ref{P-infinity} each has a limit. These two limits must be equal as (c) implies
that for any $\epsilon > 0$, there will be some $i$ such that $\Delta^{(i)} < \epsilon$. As noted by
Panconesi, (c) alone does not guarantee the Corollary; consider the case where the sequences of (a)
and (b) are identical but oscillating, say, each sequence being $((-1)^i)_{i\geq 1}$.

Now we prove Claim \ref{P-infinity}.

\begin{proof}
Proofs of (a):
\begin{align*}
m^{(i+1)}  &\defeq\min_r P_{rk}^{i+1}\\
            &=\min_r\sum_s P_{rs}P^i_{sk}\\
            &\geq\min_r\sum_s P_{rs}m^{(i)}\\
            &=m^{(i)}\min_r\sum_s P_{rs}\\
            &=m^{(i)}
\end{align*}
Thus, $m^{(i+1)}\geq m^{(i)}$, proving (a).
A proof of (b) of Claim \ref{P-infinity} can be given in a similar manner.

Proof of (c):\\
As before, our attention is on the $k$th column of the powers of the transition matrix $P$.
Consider the $rk$ entry of $P^{i+1}$, for an arbitrary $r$. As $P^{i+1} = PP^{(i)}$,
\[P^{i+1}_{rk} = \sum_l P_{rl}P^i_{lk}\]
Let a maximal entry of the $k$th column of $P^i$, viz., $M^{(i)}$, occur as the $s$th entry in
the column. Then
\begin{align*}
P^{i+1}_{rk} &= P_{rs}M^{(i)} + \sum_{l\neq s} P_{rl}P^i_{lk}\\
               &\geq P_{rs}M^{(i)} + (1-P_{rs})m^{(i)}\\
               &= m^{(i)} + P_{rs}(M^{(i)} - m^{(i)})\\
               &\geq m^{(i)} + p_{\rm min}(M^{(i)} - m^{(i)})
\end{align*}
where $p_{\rm min}$ denotes the minimum of all entries in the transition matrix $P$. (This
is positive, as $P$ is assumed to be positive.)
As the above inequality is satisfied by an arbitrary entry of the $k$th column of $P^{(i+1)}$,
it will be satisfied by $m^{(i+1)}$. Therefore, we have,
\begin{equation}\label{min-entry}
m^{(i+1)} \geq  m^{(i)} + p_{\rm min}(M^{(i)} - m^{(i)})
\end{equation}

In a similar manner, next we get an upper bound on $M^{(i+1)}$. We focus again
on an arbitrary entry of the $k$th column of $P^{(i+1)}$, this time we will take out the minimal element. 
Let a minimal entry of the $k$th column of $P^i$, viz., $m^{(i)}$, occur as the $t$th entry in
the column. Then
\begin{align*}
P^{i+1}_{rk} &= P_{rt}m^{(i)} + \sum_{l\neq t} P_{rl}P^i_{lk}\\
               &\leq P_{rt}m^{(i)} + (1-P_{rt})M^{(i)}\\
               &= M^{(i)} - P_{rt}(M^{(i)} - m^{(i)})\\
               &\leq M^{(i)} - p_{\rm min}(M^{(i)} - m^{(i)})
\end{align*}
As the above inequality is satisfied by an arbitrary entry of the $k$th column of $P^{(i+1)}$,
it will be satisfied by $M^{(i+1)}$. Therefore, we have,
\begin{equation}\label{max-entry}
M^{(i+1)} \leq  M^{(i)} - p_{\rm min}(M^{(i)} - m^{(i)})
\end{equation}

Subtracting inequality \ref{min-entry} from the inequality \ref{max-entry}, we obtain  
\[M^{(i+1)} - m^{(i+1)} \leq M^{(i)} - m^{(i)} - 2p_{\rm min}(M^{(i)} - m^{(i)})\]
That is,
\[M^{(i+1)} - m^{(i+1)} \leq (1 - 2p_{\rm min})(M^{(i)} - m^{(i)})\]
In terms of $\Delta$, we get
\[\Delta^{(i+1)} \leq (1-2p_{\rm min})\Delta^{(i)}\]
Noting that $\Delta^{(1)}\leq 1$, we get
\begin{equation}\label{exp-rate}
\Delta^{(n)} \leq (1-2p_{\rm min})^{n-1}
\end{equation}
Assuming that the state space of the Markov chain has three or
more states ensures that $0 < (1 - 2p_{\rm min}) < 1$. Thus, the
inequality \ref{exp-rate} proves that $\Delta^{(n)}$ goes ({\em exponentially fast}) to $0$ as
$n$ goes to infinity, thereby proving (c) of Claim \ref{P-infinity}.
%\footnote{We note
%that this proof of the fundamental theorem of Markov chains, a stochastic process, makes no
%appeal to probability at all. The key idea that is used is that when we take the dot product
%of a row of a positive stochastic matrix with a column of the matrix, we are performing a weighted averaging
%of the column entries-- the result will be one between the smallest and the largest entry
%of the column, and the between-ness is strict when the stochastic matrix is positive. 
%We also note that Panconesi in his proof goes from the inequality (cast in our notation)     
%$P^{(i+1)}_{rk} \geq P_{rs}M^{(i)} + (1-P_{rs})m^{(i)}$\\
%to $P^{(i+1)}_{rk} \geq p_{\rm min}M^{(i)} + (1-p_{\rm min})m^{(i)}$ in a single step,
%which may appear incorrect as the second term in the RHS of a $\geq$ is being increased in general.
%However, the step is justified as the increase in the second term is less than or equal to
%the decrease in the first term as we have the guarantee that $M^{(i)} \geq m^{(i)}$. The proof
%given here should appear clearer on that score.}    
This completes the proof of the Claim \ref{P-infinity}, which, as we have seen, establishes that
$P^\infty$ is of the form
\[\left[\begin{array}{cccc}
      \pi_1  &\pi_2 &\ldots  &\pi_n\\ 
      \pi_1  &\pi_2 &\ldots  &\pi_n\\ 
        .    &.     &\ldots  &. \\
      \pi_1  &\pi_2 &\ldots  &\pi_n 
     \end{array}\right]\]
\end{proof}

Let us now prove the remaining assertions of Proposition \ref{matrix-P-infinity}. We prove 
first that $\pi$ is a distribution. This follows from the fact that for each $i\geq 1$, $P^i$
is a stochastic matrix. The proof is by induction on $i$: suppose $P^i$ is stochastic. Consider
$P^{i+1}{\bf 1}$ where ${\bf 1}$ is the column vector of all $1$'s. 
\begin{align*}
P^{i+1}{\bf 1} &= P(P^i{\bf 1})\\
               &= P{\bf 1}\\
               &= {\bf 1}
\end{align*}
%(We get the  second and the third equality above by noting that we get ${\bf 1}$ when we take
%the dot product of a row vector which is a probability distribution with ${\bf 1}$, we get ${\bf 1}$.)
Further, for each $i$, $P^i$ remains a positive matrix as $P$ is positive. Therefore, $\pi$
is a probability distribution with full support.

Next we show that $\pi$ is a stationary distribution of $P$. Consider
\begin{align*}
P^\infty P &\defeq \lim_{i\rightarrow\infty}P^iP\\
           &=\lim_{i\rightarrow\infty}P^{i+1}\\ 
           &=\lim_{i\rightarrow\infty}P^i\\
           &\defeq P^\infty
\end{align*}
Therefore we have 

\[\left[\begin{array}{cccc}
      \pi_1  &\pi_2 &\ldots  &\pi_n\\ 
      \pi_1  &\pi_2 &\ldots  &\pi_n\\ 
        .    &.     &\ldots  &. \\
      \pi_1  &\pi_2 &\ldots  &\pi_n 
     \end{array}\right]P
= \left[\begin{array}{cccc}
      \pi_1  &\pi_2 &\ldots  &\pi_n\\ 
      \pi_1  &\pi_2 &\ldots  &\pi_n\\ 
        .    &.     &\ldots  &. \\
      \pi_1  &\pi_2 &\ldots  &\pi_n 
     \end{array}\right]\]

Which shows that 
\[[\pi_1 \pi_2\ldots \pi_n]P = [\pi_1 \pi_2\ldots \pi_n]\]
proving that $\pi$ is a stationary distribution of $P$.

Next we show that the Markov chain, started with any initial distribution $\sigma$
will reach the stationary distribution $\pi$ in the limit. Let $\sigma^{(i)}$ be as defined in 
the statement of the Proposition \ref{matrix-P-infinity}. We need to show that 
$\lim_{i\rightarrow\infty}\sigma^{(i)} = \pi$. 
This follows easily using the inductive definition of $\sigma^{(i)}$:
\begin{align*}
\lim_{i\rightarrow\infty}\sigma^{(i)} &=  \lim_{i\rightarrow\infty}\sigma P^{i-1}\\
                                      &= \sigma\lim_{i\rightarrow\infty}P^{i-1}\\
                                      &= \sigma P^\infty\\
                                      &=\pi
\end{align*}
This completes the proof of the fundamental theorem.
\end{proof}

(It will be instructive to figure out where in the proof we have made an essential use of the assumption that the chain is aperiodic.)

{\bf A simplification:} Yogesh Dahiya\footnote{As an MS student at that time, Dept. of Computer Science, IIT Kanpur.} suggested that
instead of proving the two inequalities \ref{min-entry} and \ref{max-entry} to establish (c) of Claim \ref{P-infinity},
either of the two inequalities would suffice. For example, the inequality \ref{min-entry} gives
\[m^{(i+1)} \geq m^{(i)} + p_{\rm min}\Delta^{(i)}\]  
Now, taking the limits of both the sides of the above as $i$ goes to infinity, we get
\[\lim_{i\rightarrow\infty}\Delta^{(i)}\leq 0\]
Since $\Delta^{(i)}$'s are non-negative by definition, we get 
\[\lim_{i\rightarrow\infty}\Delta^{(i)} = 0\]
thereby establishing (c) of Claim \ref{P-infinity}. Further, an exponential convergence similar to \ref{exp-rate} can
also be obtained using only either of the two inequalities:
\begin{align*}
\Delta^{(i+1)} &\defeq M^{(i+1)} - m^{(i+1)}\\
               &\leq M^{(i)} - m^{(i+1)}\\
               &\leq M^{(i)} - m^{(i)} - p_{\rm min}\Delta^{(i)}\\
               &=(1 - p_{\rm min})\Delta^{(i)}
\end{align*}
\subsection{Convergence rate as implied by the proof}
Finite Markov chains have been used for obtaining approximate solutions of $\#P$-hard counting
problems \cite{sinclair} and in combinatorial optimization problems \cite{biswas}, \cite{min-wt}. Whether the
resulting algorithms are efficient or not depends on whether or not the underlying Markov chains
come {\em close} to their stationary distributions {\em quickly}.
This issue is also of importance in a  recent application of Markov chains that models evolution 
of a population of closely related virus species with a view to design drugs  against certain diseases like AIDS \cite{nisheeth}. 
The notion of {\em mixing time} is generally used to formally capture how quickly does a Markov chain come
close to its stationary distribution. The proof of the fundamental theorem given above provides a bound
on mixing time;  we see in this section what the bound is and whether it is a good bound or not. First, we
state the necessary definitions \cite{mixing-book}. 

The {\em total variational distance} between two probability distributions $\mu$ and $\nu$ on state
space $\Omega$, denoted as $\|\mu - \nu\|_{\rm TV}$ is defined as 
\[\|\mu - \nu\|_{\rm TV}\defeq\frac{1}{2}\sum_{x\in\Omega}|\mu(x)-\nu(x)|\]
For a chain with transition probability matrix $P$, the {\em distance} $d(t)$ of the chain, after $t$ steps from the stationary distribution $\pi$, is the maximum 
of the total variational distances between $\pi$ and the distribution resulting after $t$ steps, starting the chain 
from any state of $\Omega$, that is, 
\[d(t)\defeq \max_{x\in\Omega}\|P^t(x,.)-\pi\|_{\rm TV}\]    
where $P^t(x,.)$ denotes the distribution that results after $t$ steps of starting the chain from
the state $x$. We note that the distribution $P^t(x,.)$ is the row of the matrix $P^t$ that corresponds to 
the state $x$. The {\em mixing time} within a given $\epsilon$, denoted as $t_{\rm mix}(\epsilon)$, is defined as 
\[t_{\rm mix}(\epsilon)\defeq \min\{t|\ d(t)\leq\epsilon\}\] 
{\em Mixing time} $t_{\rm mix}$ conventionally denotes the mixing time within $1/4$, i.e., $t_{\rm mix}(1/4)$.
  
The above defined $d(t)$ easily relates to $\Delta^{(t)}$ defined in Claim \ref{P-infinity}. Consider
the element $P_{ij}^t$. Fixing our attention to the column $j$, clearly, $|P_{ij}^t - \pi_j|\leq M^{(t)}-m^{(t)}$,
that is, $|P_{ij}^t - \pi_j|\leq \Delta^{(t)}$. This is true of any state $i$, therefore, $d(t) \leq n\Delta^{(t)}$, where $n$ is the number of states.

Next, for $d(t)$ to be less than or equal to $\epsilon$, it suffices to have $n\Delta^{(t)}\leq\epsilon$. In the proof of
Claim \ref{P-infinity}, we had established that
\[\Delta^{(t)} \leq (1-2p_{\rm min})^{t-1}\]
(We recall that $p_{\rm min}$ denotes the minimum entry in the (positive) matrix $P$.)
This gives us the condition 
\[t \geq \frac{1}{2p_{\rm min}}\ln \frac{n}{\epsilon}+1\Rightarrow d(t)\leq\epsilon\]  
Therefore, 
\[t_{\rm mix}(\epsilon) = O\left(\frac{1}{2p_{\rm min}}\ln \frac{n}{\epsilon}\right)\]  

The above analysis assumes the transition probability matrix $P$ to be positive. Let us consider
the general case of an ergodic Markov chain $M$ with transition probability matrix $Q$. We define $m$
as $m \defeq\min\{k\geq 1 :Q^k > 0\}$. (Ergodicity of $M$ ensures that $m$ will be defined.) After the first $m$ steps,
if we consider moves of $M$ in sequences of $m$ steps each, then actions of these move sequences are of course defined
through $Q^m$, which is a positive matrix. For the chain $M$, therefore, we have 
\begin{equation}
\label{mix-time}
t_{\rm mix}(\epsilon) = O\left(\frac{m}{2p_{\rm min}}\ln \frac{n}{\epsilon}\right)
\end{equation}  

In many applications, we would like the underlying Markov chain to be {\em rapidly mixing}. A chain is rapidly mixing
if it is the case that $t{\rm mix} = O(p(n))$ where $n$ is the length of the description of the chain, and $p(.)$ is a
fixed polynomial, that is, the chain mixing time is bounded by a fixed polynomial in the size of the input to describe the chain.
The size can of course be taken as the size of the state space. However, in many applications, the input chain is implicitly
described such that input size is logarithmic in the state space. For example, consider random walks on an $n$-dimensional
hypercube, the walk can be modeled by a Markov chain for which a state is an $n$-dimensional $0-1$ vector 
that describes the current position of the walk. The Markov chain can be described by specifying what the transitions will
be, with their probabilities, given a state. As there are only $O(n)$ neighbours of any state, the description of the
Markov chain is of size $O(n)$ whereas the state space size is $2^n$.    

Before we put the relation \ref{mix-time} to use to check rapid mixing, we need to comment upon the parameter $m$ in the
equation. Trivially, $m$ is bounded by the number of states, but in many applications where the chain state space is 
exponential in the size of its description, $m$ turns out to be polynomially bounded by the chain description. For the
hypercube random walk example, notice that $m$ is equal to $n$, as the chain has a path from any state to any other of length
$n$ or less, with paths from $[0 0\ldots 0]$ to $[1 1\ldots 1]$ being of length $n$. Similarly, for the case of card shuffle
where a move consists of picking a card from the deck uniformly at random and then placing the chosen card on top of the deck,
for a deck with $n$ cards, there are $n!$ states but we can get to any state from any other in at most $n$ moves.
Thus, for both these examples where the number of states is exponential in the (implicit) input size, $m$ is of the size of the
input, rather than of the size of the state space.

In reasonable descriptions of the input chain, the description size is no less than logarithmic in the size of the states.
Assuming such descriptions, We see from the relation \ref{mix-time} that a sufficient condition for a chain to be rapidly
mixing is that both $m$ and $\frac{1}{p_{\rm min}}$ are bounded by a fixed polynomial in the size of the chain description.
This condition immediately implies rapid mixing of random walks on a graph because $p_{\rm min}$ is no less than $1/n$ for
an $n$-vertex graph. However, the condition is too weak to prove rapid mixing in many cases. Consider the case of random
walks on $n$-dimensional hypercube. Although as we have remarked, the value of $m$ poses no problem, $p_{\rm min}$ is too small:
the probability of moving from $[0 0\ldots 0]$ to $[1 1\ldots 1]$ is $n!/n^n$. Using Stirling's approximation 
$n!\sim\sqrt{2\pi n}(n/e)^n$, we see that the probability is inverse exponential in $n$.\footnote{In order to make the chain
aperiodic, the chain is made lazy, thereby the value is further reduced by a factor of $2^n$ which is ignored for
the sake of simplicity.} Therefore, the relation \ref{mix-time}
does not prove rapid mixing of the chain, though the chain does mix rapidly, using a coupling argument, one can prove that the chain
will mix in $O(n\log n)$ steps.
  
\section{Stationarity: Existence, Convergence}
A common approach (\cite{mixing-book}, \cite{haggstrom}, \cite{mitzen}) to prove the fundamental theorem is to prove
separately the following:\\
(a) if a chain is irreducible then it has a stationary distribution in which every state has a non-zero probability,\\
(b) if $\pi$ is a stationary distribution of an ergodic chain with with transition probability matrix $P$, then starting with any
initial distribution, the chain reaches $\pi$ in the limit.

An easy consequence of (b) above is that, a stationary distribution, if it exists, is unique. (Theorem 5.3 of \cite{haggstrom}).
For proof, suppose that $\pi$ and $\pi'$ are two stationary distributions, we show that $\pi = \pi'$. Let us start the chain 
with $\pi$ as its initial distributions. Then, denoting $\pi P^n$ as $\pi^(n)$, we have from (b) that 
$\lim_{n\rightarrow\infty}\pi^(n) = \pi'$. However, as $\pi$ is stationary, $\pi = \pi P$, so,  
$\lim_{n\rightarrow\infty}\pi^(n) = \pi$. Therefore, $\pi = \pi'$.\\  
Clearly then, proving (a) and (b) together will prove the fundamental theorem.

We note first that irreducibility is both a necessary and sufficient condition for the existence of a stationary distribution. We
had seen in Section 2 the necessity, and (a) above guarantees sufficiency.
We note that the converse of (a) is not true, a trivial counter example is that of a chain with the identity matrix as
its transition probability matrix. 
Further, it is easy to see that a periodic chain can have stationary distribution: a trivial example is a chain
with state space $\{s_0, s_1\}$ and transition probabilities are given by the rule: if in state $s_i$, with probability $1$ go to
the other state, namely, $s_{1-i}$. Clearly, the underlying transition graph is bipartite, but the chain has the unique stationary 
distribution $[0.5, 0.5]$.

Actually, irreducibility guarantees not just existence, but also 
the {\em uniqueness} of the positive stationary distribution, \cite{mixing-book} (Corollary 1.17), (\cite{saloff-coste} (Lemma 1.2.2).
We briefly sketch the proof as given in $\cite{mixing-book}$ (\cite{saloff-coste} proof is essentially the same).\\
Let us call a function $h$ from the state space $\Omega$ to
$\mathbb{R}$ to be {\em harmonic} if, regarding $h$ as a column vector, we have $h = Ph$. If $P$ is the transition matrix
of an irreducible chain, then $h$ must be a constant function. The proof: since $\Omega$ is finite, there exists an 
$x_0\in\Omega$ such that $h$ is maximal at $x_0$. Let $h(x_0)$ be equal to $M$. Let $z\in\Omega$ be such that 
$P(x_0,z) > 0$ and $h(z) < M$. Now, $h$ being harmonic,
\[h(x_0) = P(x_0, \cdot)h = P(x_0,z)h(z) + \sum_{w\neq z}P(x_0,w)h(w)\]
Clearly, the RHS of the second equality is strictly less than $M$, which then gives $h(x_0) < M$, a contradiction. This implies
that $h(z) = M$. Irreducibility implies that for every $y\in\Omega$, there is some $q$ such that there is a sequence
$x_0, x_1,\ldots, x_q = y$ with $P(x_i, x_{i+1}) > 0$. Repeating the argument tells us 
that $M = h(x_0) = \ldots = h(x_{q-1} = h(y)$. Therefore, $h$ is a constant function. Now, by definition, $(P - I)h = 0$. As
$h$ is a constant (column) vector, the kernel of $(P - I)$ is $1$. Therefore, the column rank of $P - I$ is $\Omega - 1$.
As the row and the column ranks of a matrix are the same, the solution space of the matrix equation $\sigma = \sigma P$
is also one dimensional. Normalizing the solutions to have all entries of a solution summing to $1$, there will be a
unique solution to $\sigma = \sigma P$.

The following subsection deals with proving (a), and (b) is  dealt with in the next subsection.
\subsection{Existence of stationary distribution}
We provide two proofs, the first is probability based, and the other one uses only graph theoretic argument. 
Of course, linear algebra based proofs are also are there, as the one given in \cite{saloff-coste}, making 
essential use of Perron's theorem. However, since Section 5 details a complete proof of the fundamental theorem entirely
based on linear algebra and matrix analysis, which too uses Perron's theorem, we do not discuss in this subsection
linear algebra based proofs of existence of positive stationary distributions.
\subsubsection{Stationarity using probability argument}
We sketch the proof given in \cite{mixing-book}.
%
%%%We briefly sketch two other proofs of the fundamental theorem. The first that we discuss \cite{mixing-book}, \cite{haggstrom} is possibly the most 
%commonly found proof; this proof, unlike the one given in Section 2, is through probabilistic arguments.
%%\subsection{Proof using probabilistic argument}
%We sketch the proof as given in \cite{mixing-book}. 
%(a) a constructive existence proof that a stationary
%distribution exists, and then\\
%An easy consequence of \ref{conv-thm} is that there cannot be more than one stationary distribution (Theorem 5.3, \cite{haggstrom}.)
%Clearly, (a) and (b) then together prove the fundamental theorem.
%

\begin{psketch}
%{\bf Proof sketch of (a)}\\
The intuition is that the probability $\pi_y$ corresponding to a state $y$ in a stationary distribution is 
the fraction of time the chain spends in state $y$ in the ``long-term''. As the chain runs, if we 
consider successive sequences, each of which starts at (a specific but arbitrary) state $z$ and ends when the chain
revisits $z$ for the first time after the start of that sequence, then these sequences are identically distributed.
Therefore, the average number of times a state $y$ is visited per sequence should be proportional to
$\pi_y$.

Let $M$ be a Markov chain $(X_i)_{i\geq 0}$ with state space $\Omega$ and 
the transition probability matrix $P$. 
Let $z$ be an arbitrary state. We follow the noation in \cite{mixing-book} to  
denote $\prob_z(\cal{E})$, $\expect_z(Y)$ respectively as the probability of the event $\cal{E}$, and the expectation
of $Y$, when the initial distribution has the state $z$ with probability $1$. Also, define $\tilde{\pi_y}$ as\\
$\tilde{\pi_y}\defeq \expect_z($ number of visits to $y$ before returning to $z$ for the first time$)$.
Let $\tau^+_z$ denote $\min\{t\geq 1|X_0 = z, X_t = z\}$, i.e., the first return time to $z$. Let $I_t$ denote the indicator
variable which is $1$ if $X_t = y$ and $\tau_z^+ > t$. Then, 
\[\tilde{\pi_y} = \expect(\sum_{t=0}^\infty I_t) = \sum_{t=0}^\infty\prob_z(X_t = y, \tau^+_z>t)\]  
Next, we verify that $\tilde{\pi}$ is stationary: we show that for arbitrary $y\in\Omega$,\\
$\tilde{\pi}_y = \sum_{x\in\Omega}\tilde{\pi}_x P(x,y)$. Using the definition of $\tilde{\pi}_x$,
\begin{equation}
\label{lhs}
\sum_{x\in\Omega}\tilde{\pi}_xP(x,y) = \sum_{x\in\Omega}\sum_{t=0}^\infty\prob_z(X_t = x, \tau_z^+ > t)P(x,y)
\end{equation}
We use below the fact that the event $\tau_z^+ \geq t+1$, i.e., $\tau_z^+ > t$
is determined entirely by $X_0, X_1,\ldots, X_t$ in obtaining the third equality in the following:
\begin{align*}
&\prob_z(X_t = x, X_{t+1} =y, \tau_z^+ \geq t+1)\\ 
&= \prob_z(\tau_z^+ \geq t+1 | X_t = x, X_{t+1} = y)\prob_z(X_t = x, X_{t+1} = y)\\
&= \prob_z(\tau_z^+ \geq t+1 | X_t = x)\prob_z(X_t = x, X_{t+1} = y)\\
&= \prob_z(\tau_z^+ \geq t+1 | X_t = x)\prob_z(X_t = x) P(x,y)\\
&= \prob_z(\tau_z^+ \geq t+1, X_t = x) P(x,y)
\end{align*}
Reversing the order of summation in \ref{lhs}, and using the above,
\begin{align*}
&\sum_{x\in\Omega}\tilde{\pi}_xP(x,y)\\ 
&= \sum_{t=0}^\infty\sum_{x\in\Omega}\prob_z(X_t = x, X_{t+1} = y, \tau_z^+ > t)\\
&= \sum_{t=0}^\infty\prob_z(X_{t+1} = y, \tau_z^+ \geq t+1)\\
&= \sum_{t=1}^\infty\prob_z(X_t = y, \tau_z^+ \geq t)
\end{align*}

The right hand side of the above syntactically is almost the same as in the definition of $\tilde{\pi}_y$, in fact
with a little work it can be shown that the right hand side of the above {\em is} $\tilde{\pi}_y$, which establishes that
$\tilde{\pi}$ is stationary: $\tilde{\pi}P = \tilde{\pi}$. To obtain a probability distribution from $\tilde{\pi}$,
we normalize it by $\sum_x\tilde{\pi}_x$, which is $\expect_z(\tau^+_z)$. Thus, the distribution is $\pi$
where 
\[\pi_x = \frac{\tilde{\pi}_x}{\expect_z(\tau_z^+)}\] 
As $z$ was arbitrary, we can set $z$ to be $x$, to get
\[\pi_x = \frac{1}{\expect_x(\tau_x^+)}\]
We note that, by construction, each $\pi_x > 0$. We also note that proof above does not
need the chain to be aperiodic, though irreducibility is needed--  we have made use of the
tacit assumption that $\tilde{\pi}_x < \infty$ for all states $x$ which would not hold for a reducible
chain.
\end{psketch}

\subsubsection{Stationarity using graph theoretic argument}
The proof we give here is based on one given in \cite{rajeeva} which is the unpublished notes on Markov chains by Rajeeva Karandikar.\footnote{Karandikar 
communicated to us that his proof was adapted from \cite{freidlin-wentzell}.}

%\begin{proof}
As before, let $\Omega$ be the state space and $P$ be the transition probability
matrix of a Markov chain $M$, which is assumed to be irreducible. For $x, y\in\Omega$,
we use $p_{xy}$ as an abbreviation of $P(x,y)$. 

Suppose we are able to define a $\gamma:\Omega\rightarrow\mathbb{R}^+$,
satisfying, for every $y\in\Omega$ 
\[\sum_{x\in\Omega}\gamma(x)p_{xy} = \gamma(y)\]
then $\gamma$ is stationary for $P$ (though not necessarily a distribution). The equation is above is equivalent to
\begin{align*}
\sum_{x\in\Omega, x\neq y}\gamma(x)p_{xy} &= \gamma(y) - \gamma(y)p_{yy}\\
                                          &= (1 - p_{yy})\gamma(y)\\
                                          &= \sum_{x\in\Omega, x\neq y}p_{yx}\gamma(y)
\end{align*} 
That is, we need to establish, for each $y\in\Omega$,
\begin{equation}
\label{balance-condn}
\sum_{x\in\Omega, x\neq y}\gamma(x)p_{xy} = \sum_{x\in\Omega, x\neq y}\gamma(y)p_{yx}
\end{equation}
Intuitively, the equation captures a balance condition, {\em viz.}, attaining stationarity means that for every $y$, 'flow' into $y$ equals the 'flow' out of $y$.
 
Such a $\gamma$ is defined from $G_M$, the underlying directed graph capturing $P$. $G_M$
has $\Omega$ as its set of verices, and for every pair of vertices, $x,y$, there is an edge from $x$ to $y$ of weight $p_{xy}$.
For a vertex $x$, let $\tau(x)$, a subgraph of $G_M$, be called an {\em upward spanning tree rooted at $x$}, if
$\tau(x)$ satisfies the following:
\begin{itemize}
\item $\tau(x)$ spans $G_M$: every vertex of $G_M$ occurs once and only once in $\tau(x)$,
\item From every vertex in $\tau(x)$, except $x$, there is exactly one outgoing edge, and the outdegree
of $x$ is zero.  
\item From every vertex of $y$ of $G_M$ other than $x$, there is one and exactly one path from $y$ to $x$
in $\tau(x)$.
\end{itemize}
For every vertex $x$, we define $\mathcal{T}(x)$ as
\[\mathcal{T}(x)\defeq\{\tau(x)\ |\ \tau(x)\mbox{ is an upward spanning tree rooted at } x\}\]
The weight $W(H)$ of a subgraph $H$ is defined to be the product of the weights of the edges in $G$, and the weight
of a set $S$ of subgraphs is the sum of the weights of its elements. In particular,
the {\em  weight} $W(\mathcal{T}(x)$ is the sum of the weights of the upward spanning trees in $\mathcal{T}(x)$,
where the weight of the upward spanning tree $\tau(x)$ is the product of its edge weights.
Thus, 
\[W(\mathcal{T}(x)\defeq\sum{\tau(x)\in\mathcal{T}(x)}\prod_{(y,z)\in\tau(x)}p_{yz}\] 

For each $x\in\Omega$, we define $\gamma(x)$ as
\[\gamma(x)\defeq W(\mathcal{T}(x)\]  
Clearly, the Markov chain being irreducible, for every $x\in\Omega$, there is at least one upward spanning tree
rooted at $x$, so, $\gamma(x)$ is positive. Next, 
\begin{claim}
$\gamma(x)$'s, as defined above satisfy the stationarity condition Eqn. \ref{balance-condn}.
\end{claim}
%\begin{proof}
A proof of the above is as follows.\\
The LHS of the balance condition, Eqn. \ref{balance-condn}, can be seen as the weight of the following
set $A$ of subgraphs of $G_M$: $A\defeq\cup_{x\in\Omega, x\neq y, \tau(x)\in\mathcal{T}(x)}(\tau(x)\cup\{(x,y)\})$.
Similarly, the RHS of Eqn. \ref{balance-condn} can be seen as the weight of the following
set $B$ of subgraphs of $G_M$: $B\defeq\cup_{x\in\Omega, x\neq y, \tau(y)\in\mathcal{T}(y)}(\tau(y)\cup\{(y,x)\})$.
We show that the two sets, $A$ and $B$ are same, because each is contained in the other. To see that $A\subseteq B$,
consider an element of $A$, which is, for some $x, x\neq y$, an upward spanning tree rooted at $x$, say $\tau'(x)$, 
with the edge $(x,y)$ added to it. Since $\tau'(x)$ is spanning, the vertex $y$ occurs in $\tau'(x)$, and let
$z$ be the vertex that immediately follows $y$ in the unique directed path 
from $y$  to $x$ in $\tau'(x)$.  By removing the edge $(y,z)$ from $\tau'(x)$ and adding instead
the edge $(x,y)$, we get a subgraph which is an upward spanning tree now rooted at $y$. The union of this new upward
spanning tree with the edge $(y,z)$ is clearly a member of $B$, being the subgraph which is the union
of an upward spanning tree rooted at $y$ along with the edge $(y,z)$. To check the containment of $B$ in $A$, consider an element
of $B$, which is some upward spanning tree rooted at $y$, say $\tau''(y)$ along with an edge $(y,w)$ for some $w, w\neq y$.
Consider the unique path from $w$ to $y$ in $\tau''(y)$, and let $u$ be the vertex which immediately precedes $y$ in this
path. By deleting the edge $(u,y)$ from $\tau''(y)$ and adding instead the edge $(y,w)$, we get an upward spanning tree
now rooted at $u$, say, $\tau'''(u)$. This tree, along with the edge $(u,y)$, is indeed a member of $A$. Karandikar
gives a succinct description of the set $A (= B)$: it is the set of all minimally spanning sets with exactly one cycle
that contains $y$ and in which every vertex has outdegree exactly $1$.    

As $A = B$, they will have identical weights. This ends the proof of the above claim.
%\end{proof}     

Defining $\pi(x)$ as $\gamma(x)\sum_{y\in\Omega}\gamma(y)$, we get a stationary distribution for the
Markov chain $M$.
%\end{proof}
  
The proof above uses irreducibility, but aperiodicity has not been assumed.
%{\bf Proof  of (b), the Convergence theorem}\\

\subsection{Proofs of convergence to stationarity} 
As stated earlier, what we wish to prove here is: if an ergodic Markov chain $M$ admits a stationary distribution $\pi$ then
the chain attains $\pi$ in the limit, starting from any distribution. 
%We recall that when $M$
%is ergodic, then there is a finite $k$ such that $P^k$ is a positive matrix, $P$ being the transition probability
%matrix of $M$. Therefore, for proving the result, we lose no generality by assuming the transition probability matrix 
%of the chain to be positive, and we use this assumption in the proofs discussed below.
Formally, we prove:
\begin{theorem}
\label{conv-thm}
\[\lim_{n\rightarrow\infty}\|P^n(x,.)-\pi\|_{\rm TV} = 0\] for all $x\in\Omega$.
or, equivalently, $\lim_{n\rightarrow\infty} P^n(i,j) = \pi_j$ for all $i,j\in\Omega$
\end{theorem}
We discuss here two types of proofs of the above, one uses coupling, the other type uses matrix analysis.

\subsubsection{Convergence using coupling}
Coupling is a technique which has been widely used in establishing upper bounds on the mixing times of Markov chains \cite{mixing-book},
\cite{mitzen}. It is therefore not surprising that this technique can be used for proving the convergence result, which is a weaker
result than proving, given any $\epsilon\geq 0$, starting from an arbitrary initial distribution, an upper bound on the number of
steps a chain will take to come $\epsilon$-close to its stationary distribution, which is what is used to establish an upper
bound on mixing time of the chain.

We first explain the notion of coupling and indicate how the notion is used in establishing upper bounds on mixing time. We
shall define coupling of a Markov chain in a manner different from the usual ones, say, 
as found in \cite{mixing-book}, and refer to relevant literature as to why our definition
is more appropriate. Actually, the coupling notion used in practice is a restricted notion called the {\em faithful coupling}, we shall 
briefly indicate why the restricted notion is used rather than the general notion as given in the \cite{mixing-book}

Next, we prove Theorem \ref{conv-thm} and the proof we shall provide, is by Karandikar, given in \cite{rajeeva}. Although it is
a coupling based proof, it assumes no knowledge of coupling, and hence can be read independently without first going through
the next two subsections.

\subsubsection{Coupling of two distributions}
Let $\mu$ and $\nu$ be two distributions on the same set, say, $\Omega$. A coupling of these two distributions is a pair $(X, Y)$ of
random variables $X$ and $Y$, defined on the same probability space, therefore, one can define a joint distribution of $X$ and $Y$.
Let $\theta$ be such a joint distribution, (clearly, $\theta$ is a distribution on $\Omega\times\Omega$). For $(X, Y)$ to be a 
coupling of $\mu$ and $\nu$, $\theta$ needs to satisfy the property that the marginal distribution of $X$ in $\theta$ is $\mu$ and the 
marginal distribution of $Y$ is $\nu$.

For the same pair of distributions, in general, there will be many couplings, i.e., there will be many joint distributions of two 
random variables $(X, Y)$ such that the marginal distribution of $X$ is $\mu$ and the marginal distribution  of $Y$ is $\nu$. For
examples, we refer to the discussion preceding Proposition 4.7 of \cite{mixing-book}. Also, there is an example given in the 
\cite{coupling-paper} of two distributions with exactly one coupling.

The importance of coupling of two distributions stems from the following fact: if $(X, Y)$ is a 
coupling of two distributions $\mu$ and $\nu$, then
$\prob(X\neq Y)$ is an upper bound of the total variational distance between $\mu$ and $\nu$, $\|\mu - \nu\|_{\rm TV}$. Further,
the infimum of all couplings is exactly the total variational distance, and this coupling, called the {\em optimal coupling}, can be
defined constructively, given any two distributions. We refer to Proposition 4.7 of \cite{mixing-book} for a proof.

\subsubsection{Coupling of a Markov chain}

Let $M$ be a Markov chain with its state space as $\Omega$, and transition matrix $P$. We define the coupling notion of $M$ in the 
following way:
a coupling of $M$ is given by an infinite sequence $(X_i, Y_i)_{i\geq 0}$ of pairs of random variables, 
each taking values from $\Omega$, and the sequence
satisfies the property that for some two distributions $\mu_0$  and $\nu_0$, both on $\Omega$, $(X_i, Y_i)$ is a coupling of $\mu_0P^i$
and $\nu_0P^i$, for all $i\geq 0$. Such a coupling can be constructed by defining a transition matrix $Q$ and a joint distribution
$\theta_0$ of
$\mu_0$ and $nu_0$, such that denoting, for all $i\geq 0$, $\theta_i$ as $\theta_0Q^i$, then $\theta_i$ is a coupling of $\mu_0P^i$ and $\nu_0P^i$.
Thus the distribution of $X_i$ is $\mu_0P^i$ and the distribution of $Y_i$ is $\nu_0P^i$, for all $i\geq 0$. Therefore, both $(X)_{i\geq 0}$,
and $(Y_i)_{i\geq 0}$ are two evolutions of the chain $M$, for certain initial distributions, but crucially, what we are {\em not}
saying is that  
$(X)_{i\geq 0}$ and
$(Y_i)_{i\geq 0}$ are Markov chains.

If we compare the definition above with the one given in \cite{mixing-book}, ours appears unnecessarily complicated, as our definition is 
in terms of specific initial distributions for the two copies of the chain, as well as it is in terms of a specific joint distribution
of the two initial distributions of the Markov chain. The reason why our definition is the appropriate definition is given in 
\cite{hirscher}, \cite{coupling-paper}.

In any case, a coupling of a Markov chain $M$ can be seen as a process $(X_i, Y_i)_{i\geq 0}$ where each of $(X_i)_{i\geq 0}$ and 
$(Y_i)_{i\geq 0}$ follows the evolution of the chain $M$. We say that such a coupling has {\em coupled at time $T$} if $X_T = Y_T$,
further we say that the coupling has the {\em now-equals-forever} property \cite{rosenthal} 
if for any $j, X_j = Y_j$, then for every $k\geq j, 
X_k = Y_k$. One way which has been used for turning a coupling $(X_i, Y_i)_{i\geq 0}$ into a coupling with the
now-equals-forever property is by replacing $(Y_i)_{i\geq 0}$ by $(Z_i)_{i\geq 0}$ where 
\begin{equation}
\label{sticking}
Z_i = \left\{ \begin{array}{ll}
       Y_i  & \mbox{ if $i \leq T$}\\
       X_i  & \mbox{ otherwise}
              \end{array} 
       \right.
\end{equation}
where $T\defeq\inf\{i|X_i = Y_i\}$.
In other words, the process $(Y_i)_{i\geq 0}$ starts following the process $(X_i)_{i\geq 0}$ once the coupling has happened. 
\cite{hirscher} called replacing $Y_i$s by $Z_i$s as the {\em sticking} operation.

Rosenthal \cite{rosenthal} proved that the sticking operation {\em does not guarantee} that the two processes  $(X_i, Y_i)_{i\geq 0}$ and
$(X_i, Z_i)_{i\geq 0}$ will be equivalent. Rosenthal then proved that if the coupling is what he defined as {\em faithful}, then the two
processes are guaranteed to be equivalent. 

If the coupling has (or made to have through sticking) the now-equals-forever property,
then it can be shown that the following holds:
\begin{lemma}
\label{coupling-lemma}
\[\|\mu_0 P^i - \nu_0 P^i\|_{\rm TV}  \leq \prob(T > i)\]
\end{lemma}
where $T$ is as defined above, and (as we defined in the beginning of our discussion on the Markov chain coupling),
$\mu_0$ and $\nu_0$ are the two initial distributions, and $P$ is the transition matrix of the Markov chain being coupled. This follows
from Lemma 11.2 of \cite{mitzen}, the lemma is known as the {\em coupling lemma}. For an explicit proof of \ref{coupling-lemma}, we 
refer to \cite{coupling-paper}. 

Let us now indicate how coupling is used to obtain an upper bound on the mixing time. Suppose $\mu_0$ 
is the stationary distribution of $M$, say $\pi$, and let $\nu_0$ be an 
{\em arbitrary} initial distribution. By providing an upper bound on $\prob(T > i)$ for some $i$, we get an upper bound on the
variational distance between the stationary distribution and the distribution obtained after running the chain for $i$ steps. Therefore, 
given an $\epsilon\geq 0$, we can determine an 
$i$ such that $\prob(T > i)\leq\epsilon$, which will provide an upper bound on the number of steps
the chain, starting with an arbitrary distribution, needs to run so that the chain distribution comes $\epsilon$-close
to the stationary distribution. Of course, how good the bound is will depend on the coupling that is defined.

For the sake of completeness,
we give the definition of faithful coupling:
\begin{definition}[Faithful coupling]
Let $(M_i)_{i\geq 0}$ be a Markov chain $M$ with the state space $\Omega$. Then a coupling of $M$, $(X_i, Y_i)_{i\geq 0}$ is
a faithful coupling if $(X_i, Y_i)_{i\geq 0}$ is itself a Markov chain on state space $\Omega\times\Omega$ satisfying:
	\[\prob(X_{t+1} = x'|(X_t, Y_t) = (x, y)) = \prob(M_{t+1} = x'|M_t = x)\]
	\[\prob(Y_{t+1} = y'|(X_t, Y_t) = (x, y)) = \prob(M_{t+1} = y'|M_t = y)\]
for all $t\geq 0$ and for all $x, y, x', y'\in\Omega$.
\end{definition}
In terms of the matrix $Q$ defined earlier, the conditions above is equivalent to: 
that the transition matrix $Q$ satisfies the following:\\ 
for all $i, j, i', j'\in\Omega$,
\[\sum_{j'\in\Omega}Q((i,j), (i',j')) = P(i,i'),\mbox{ and }\]
\[\sum_{i'\in\Omega}Q((i,j), (i',j')) = P(j,j')\] 
Recall that $P$ is the transition matrix of the Markov chain $M$. Faithfulness appears to be a natural property, all
coupling examples in \cite{mixing-book} are said to be faithful.

\subsubsection{An coupling based proof of convergence}

We provide here a coupling based proof of the convergence theorem, the proof we provide is by Karandikar cite{rajeeva}.
\footnote{Karandikar informed us that his proof is adapted from a proof given in \cite{bill}} Although the proof is coupling
based, the proof assumes no background knowledge of coupling, it develops all the necessary concepts from the first principles. 
Such a 'simple' coupling based proof was possible because the coupling used is the trivial coupling where two independent copies
of a chain are coupled. The discerning reader will see that the proof indeed uses the concepts etc. we discussed in the last subsection.
{\em These connections are pointed out in remarks which are given in parentheses, these parenthetical remarks are not necessary
for understanding the proof.}

Let $M$ be a Markov chain with state space $\Omega$ and transition matrix $P$. $M$ is assumed to be irreducible and aperiodic. 
Let $\pi$ be a stationary probability of of $M$. We prove convergence to $\pi$ by showing that for all $i,j\in\Omega$,\\ 
$\lim_{n\rightarrow\infty}P^n(i,j) = \pi_j$\\
Let $(X_i)_{i\geq 0}$ and
$(Y_i)_{i\geq 0}$ be two independent copies of $M$, that is,\\ 
for all $i,i,j,j'\in\Omega$, $\prob(X_i=i', Y_j=j') = \prob(X_i=i')\prob(Y_j=j')$\\
Consider now $N = (X_i, Y_i)_{i\geq 0}$. (Clearly, $N$ defines a coupling of the simplest kind, called an {\em independent coupling}
which can easily shown to be a faithful coupling.)

Using independence, it is easy to prove that
\begin{itemize}
\item $N$ is a Markov chain with state space $\Omega\times\Omega$,
\item the transition matrix of $N$ is given by $Q$ where for all $i,j,k,l\in\Omega, Q((i,k),(j,l) = P(i,j)\times P(k,l)$,
\item $N$ is aperiodic and irreducible.\footnote{The chain $N$ will not be irreducible unless $M$ is both aperiodic and irreducible.
The proof below will break down if $N$ is not irreducible. This is where we make essential use of ergodicity of $M$. We are grateful to
Rajeeva Karandikar for pointing this out to us.}
\end{itemize}
%\item 

Let $t\in\Omega$ be any fixed state of $M$. Because the chain $N$ is irreducible, for any arbitrary $i,j\in\Omega$, with $X_0 = i$
and $Y_0 = j$, there will be infinitely many finite $n$'s such that $X_n = t$ and $Y_n = t$. Let $\tau$ denote the least of these $n$'s.

We define a new sequence of random variables
$(Z_i)_{i\geq 0}$ where 
\begin{equation}
\label{karandikar-sticking}
Z_i = \left\{ \begin{array}{ll}
       Y_i  & \mbox{ if $i \leq \tau$}\\
       X_i  & \mbox{ otherwise, that is, for all $i > \tau$}
              \end{array} 
       \right.
\end{equation}
(This sticking operation turns the chain $N$ to a new chain which has now-equals-forever property.)\\
\begin{claim}
\label{z-chain-markov}
$(Z_i)_{i\geq 0}$ is a Markov chain with transition matrix $P$. 
%Further, if $Y_0$ distribution is $\mu$ then $(Z_i)_{i\geq 0}$
%is, using the notation of \cite{norris}, Markov$(\mu, P)$.
\end{claim}
\begin{proof}
The Claim is proved by showing that for any $n\geq 0$,  and for all $i_0, i_1,\ldots, i_n$, where each of the $i$s is in $\Omega$,
\begin{equation}
\label{karand-1}
\prob(Y_0=i_0,\ldots, Y_n=I_n) = \prob(Z_0=i_0,\ldots, Z_n=i_n)
\end{equation}
Proving the equality establishes the Claim is because $(Z_i)_{i\geq 0}$ behaves exactly the same way 
as $(Y_i)_{i\geq 0}$ does which we know to be a Markov chain.
For a more formal argument, one can invoke Theorem 1.1.1 of \cite{norris}.
In turn, it suffices to prove that for all $m\geq 0$,
\begin{equation}
\label{karand-2}
\prob(Y_0=i_0,\ldots, Y_n=I_n, \tau=m) = \prob(Z_0=i_0,\ldots, Z_n=i_n, \tau=m)
\end{equation}
We note that if $m\geq n$ then \ref{karand-2} is true by definition of $(Z_i)_{i\geq 0}$. On the other hand, if $n > m$ and if 
$i_m\neq t$, then both LHS and RHS of \ref{karand-2} are zero. Therefore, the only interesting case to be considered is when
$n > m$ and $i_m = t$. In that case, the RHS of \ref{karand-2} is then
\[\prob(Z_0=i_0,\ldots, Z_n=i_n, \tau=m)\]
Using the definition of $Z$, the above is equal to
\[Y_0=i_0,\ldots, Y_m=t, \tau=m, X_{m+1}=i_{m+1},\ldots, X_n=i_n)\]
Using independence of $X_i$'s and $Y_i$'s, we get\\
$\prob(Y_0=i_0,\ldots, Y_m=t, \tau=m)\times\prob(X_{m+1}=i_{m+1},\ldots, X_n=i_n|X_m=t)$\\
Now, using the fact that, by definition, if $\tau = m$, then $X_m = Y_m = t$, we get from the above  
$=\prob(Y_0=i_0,\ldots, Y_m=t, \tau=m)\cdot p_{t,m+1}\cdots p_{n-1,n}$\\
The LHS of \ref{karand-2} can, a little more easily, be proved to be the same expression as above. Hence, $(Z_i)_{i\geq 0}$ is
also a Markov chain with $P$ as its transition matrix, being identical to $(Y_i)_{i\geq 0}$, which proves the 
Claim \ref{z-chain-markov}. 
\end{proof}

(Because the coupling used here is the simplest, independent coupling, the independence of the two chains $(X_i)_{i\geq 0}$
and $(Y_i)_{i\geq 0}$ made the proof above quite easy. For a general faithful coupling, such a proof, somewhat more complex,
can be found in \cite{rosenthal}, \cite{coupling-paper}.)

Finally, we are ready to prove convergence by showing that in $P^n$ as $n$ tends to infinity, each column entry for any $k\in\Omega$
has the same value $\pi_k$, the $k$th component of the stationary distribution $\pi$ which was assumed to exist for $P$.

We use the following notation: for any event $A$, and for any $i,j\Omega$, $\prob_{i,j}(A)$ will denote $\prob(A|X_0=i, Y_0=j)$,
or equivalently, $\prob(A|X_0=i, Z_0=j)$. When $P$ is a transition matrix, we recall that the $(i,j)$th entry of the matrix 
$P^n$ is denoted by $p^{(n)}_{ij}$. For the Markov chain $M$ as $P$ its transition matrix, we know that\\
$p^{(n)}_{jk} = \prob_{ij}(Y_n = k) = \prob_{ij}(Z_n = k)$. Therefore,

\begin{align*}
|p_{ik}^{(n)} -p_{jk}^{(n)}|\\ 
&= |\prob_{ij}(X_n = k) - \prob_{ij}(Z_n = k)|\\
&= |\prob_{ij}(X_n=k, Z_n\neq k) + \prob(X_n=k, Z_n=k)\\
&- \prob_{ij}(Z_n = k,X_n=k) + \prob_{ij}(Z_n=k, X_n\neq k)|\\
&= |\prob_{ij}(X_n=k, Z_n\neq k) - \prob(Z_n=k, X_n\neq k)|\\
&\leq \prob_{ij}(X_n=k, Z_n\neq k) + \prob_{ij}(Z_n=k, X_n\neq k)\\
&\leq \prob_{ij}(X_n \neq Z_n) = \prob_{ij}(\tau > n)
\end{align*}
We have already argued that the chain $(X_i, Y_i)_{i\geq 0}$, that is $(X_i, Z_i)_{i\geq 0}$, is irreducible, and therefore,
for any $i,j\in\Omega$, $\tau$ is finite. Therefore,
\[\lim_{n\rightarrow\infty} |p^{(n)}_{ik} - p^{(n)}_{jk}| \leq \prob_{ij}(\tau > n) = 0\]
From the above, we conclude that $\lim_{n\rightarrow\infty} |p^{(n)}_{ik} - p^{(n)}_{jk}| = 0$\\
Since the limit is zero, we the above is equivalent to
\begin{equation}
\label{karand-3}
\lim_{n\rightarrow\infty} (p^{(n)}_{ik} - p^{(n)}_{jk}) = 0
\end{equation}
This shows that, in the limit $n$ tending to infinity, all entries in any column $k$ of $P^n$ will have all its entries identical. Next,
we show that this value is $\pi_k$, the $k$th entry of the of $\pi$, the stationary distribution of the chain $M$ the existence of
which we have assumed.
As $\pi$ is a stationary distribution for the chain $M$, $\pi P^n = \pi$ for all $n\geq 0$. Therefore,\\
$\sum_{i\in\Omega}\pi_ip^{(n)}_{ik} = \pi_k$, the $k$th component of $\pi$.\\
Now consider
\begin{equation}
\label{karand-4}
	\sum_{i\in\Omega}\pi_i(p^{(n)}_{ik} - p^{(n)}_{jk}) = \pi_k - p^{(n)}_{jk}
\end{equation}
Using Equation \ref{karand-3}
\begin{equation}
\label{karand-5}
\lim_{n\rightarrow\infty}\sum_{i\in\Omega}\pi_i(p^{(n)}_{ik} - p^{(n)}_{jk}) = 0
\end{equation}
Using Equations \ref{karand-4} and \ref{karand-5}, we get the result we wanted to prove:
\[\lim_{n\rightarrow\infty}(\pi_k - p^{(n)}_{jk}) = 0\]
Further, the proof also shows that the stationary distribution is unique, as 
$\lim_{n\rightarrow\infty}p^{(n)}_{jk}$ will be a unique value.

\subsubsection{Convergence through matrix analysis}
%\begin{proof}[Convergence theorem]
First proof:\\
We consider first the proof given in \cite{mixing-book} \footnote{a similar proof is there in \cite{saloff-coste}} and provide the proof idea.
Let $\pi$ be a stationary distribution for $P$
and let $\Pi$ be (as in Section 2) the square matrix each row of which is $\pi$. First, we note
\begin{fact}
\label{pi-fact}
(a) $M\Pi = \Pi$ holds for any stochastic matrix $M$, and
(b) $\Pi M = \Pi$ holds for any stochastic matrix $M$ for which $\pi$ is a stationary distribution.
\end{fact}
As $P$ is positive, there exists a $\delta$ strictly between $0$ and $1$ such that for every 
pair $x, y$, $P(x,y) \geq \delta\Pi(x,y)$ holds. Let $\theta \defeq 1 - \delta$. A key ingredient of
the proof is to define a stochastic matrix $Q$ through the equation
\begin{equation}
P = (1 - \theta)\Pi + \theta Q
\end{equation}
The above gives
\begin{equation}
\label{error-eq}
P - \Pi = \theta(Q - \Pi)
\end{equation}                                                
The left hand side of the above is the ``error'' to begin with. This error reduces
exponentially as we power $P$; making crucial use of Fact \ref{pi-fact}, we prove
that for all $n \ geq 1$
\begin{equation}
\label{n-error-eq}
P^n - \Pi = \theta^n(Q^n - \Pi Q^{n-1})
\end{equation}
The proof is by induction, let us show the induction step. Suppose we have for some $k$ 
\[P^k - \Pi = \theta^k(Q^k - \Pi Q^{k-1})\]
We post-multiply the two sides of the above by the two sides of (\ref{error-eq}), and then simplify
using Fact \ref{pi-fact}:
\begin{align*}
(P^k - \Pi)(P - \Pi) &= \theta^k(Q^k - \Pi Q^{k_1})\theta(Q - \Pi)\\
P^{k+1} -P^k\Pi - \Pi P + \Pi^2 &= \theta^{k+1}(Q^{k+1} - Q^k\Pi - \Pi Q^k + \Pi (Q^{k-1}\Pi)\\
P^{k+1} - \Pi        &= \theta^{k+1}(Q^{k+1} - \Pi - \Pi Q^k + \Pi^2)\\
P^{k+1} - \Pi &= \theta^{k+1}(Q^{k+1} - \Pi Q^k)
\end{align*}
This proves the induction step.
In the derivation above, besides Fact \ref{pi-fact}, we have also used the fact that the product of two
stochastic matrices is also stochastic, and therefore, any power of a stochastic matrix is also stochastic.
Also, for a matrix $Q$, we have taken $Q^0$ by definition to be the identity matrix $I$.

The last step is to consider the $x$th row of of the resultant matrix on each side  of
\[P^n - \Pi = \theta^n(Q^n - \Pi Q^{n-1})\]
summing the absolute values of the elements on each side, then dividing by two to get
\[\|P^n(x,.) - \pi\|_{\rm TV} = \theta^n\|Q^n(x,.) - \Pi Q^{n-1}(x,.)\|_{\rm TV}\]
Noting that $1$ is the largest value that a total variational distance can take, we get 
\[\|P^n(x,.) - \pi\|_{\rm TV} \leq \theta^n\]
This completes the proof sketch of the Convergence theorem.
%\end{proof}

Second proof that uses an interesting matrix norm\\

\section{Linear algebra proof of the Fundamental Theorem}
The proof we provide makes use of Perron's theorem of 1907, a result that applies to all
real, positive square matrices. Perron's result was (somewhat weakly) extended to 
real, non-negative square matrices by Frobenius in 1912. Perron-Frobenius results have
many applications, we refer to \cite{maccluer} for a survey. The proof below essentially details
the proof sketch given there.
 
\begin{proof}
The statement of Perron's theorem is:
\begin{theorem}[(Perron)]
Let $A$ be a positive real square matrix. The largest eigenvalue $\lambda$ of $A$ is real,
with algebraic (and therefore, geometric) multiplicity of $1$, and with an associated
eigenvector which is both real and positive. All other eigenvalues of $A$ are strictly smaller than 
$\lambda$ in absolute value.
%Let $\rho(A)$, the spectral radius of $A$, denote
%the maximum of the absolute values of $A$. Then:\\
%(a) $\rho(A)$ is an eigenvalue, and it has a positive eigenvector.\\
%(b) $\rho(A)$ is the only eigenvalue on the disc $|\lambda| = \rho(A)$.\\
%(c) $\rho(A)$ has algebraic (and therefore, geometric) multiplicity $1$.
\end{theorem}
Let $M$ be a finite ergodic Markov chain with $\Omega$ as its set of states, and $P$ as its transition probability matrix. Without loss
of generality, we assume $P$ to be positive. We note that $1$ is an eigenvalue of $P$, because $P{\bf 1} = {\bf 1}$ as $P$ is stochastic.
In fact, we show next that $1$ is the largest real eigenvalue of $P$. Suppose otherwise, let $\lambda > 1$ be the largest real eigenvalue of $P$.
As the spectra of $P$ and its transpose $P^T$ are same, $\lambda$ is the largest real eigenvalue of $P^T$ as well. Perron's theorem says
that there will be a positive eigenvector corresponding to $\lambda$, let this eigenvector, after normalization so that its components add to $1$,
be\footnote{We remind that,  as per the non-standard convention we are following here, $\mu$ denotes a row vector.} $\mu^T$. Then,
$P^T\mu^T = \lambda\mu^T$. Taking transposes, we get $\mu P = \lambda\mu$. Thus, $P$ transforms a probability distribution to 
something which is not a probability distribution. This is a contradiction because $P$, being stochastic, always transforms probability
distributions to probability distributions: suppose  $\rho$ is a probability distribution on $\Omega$ then so is
$\sigma \defeq\rho P$, for, $\sum_{x\in\Omega} = 1$ as\\ 
$\sigma\cdot{\bf 1} = \rho(P{\bf 1}) = \rho\cdot{\bf 1} = 1$, and no element of $\sigma = \rho P$ can be negative as $P$ is
positive, and $\rho$, being a probability distribution, cannot have any negative element. Therefore, $1$ is the largest real eigenvalue
of $P$. 

We have it from Perron's theorem that the eigenvalue $1$ is of multiplicity one, and it strictly dominates all other eigenvalues.  
Casting $P$ in Jordan canonical form, $P = MJM^{-1}$, $J$ being a Jordan matrix where the first Jordan block
corresponds to the largest eigenvalue of $P$, namely, $1$. Since this eigenvalue is of multiplicity $1$, the first
block of $J$ consists of a single element, namely, $1$. Next, we show that
%Using the Jordan canonical form for $P$, let $P = MJM^{-1}$, where $J$
%being the Jordan matrix in which $1$ is placed as the first Jordan block, $1$ being an eigenvalue of multiplicity $1$.
%
%The remaining task is to prove that the chain , starting with any initial distribution, will converge to $\pi$. 
%Towards this, first we prove that 
\begin{equation}
\label{j-infinity}
J^{\infty}\defeq\lim_{n\rightarrow\infty}J^n = \left[\begin{array}{cccc}
      1  &0    &\ldots  &0\\ 
      0  &0    &\ldots  &0\\ 
        .    &.     &\ldots  &. \\
      0  &0  &\ldots  &0 
     \end{array}\right]
\end{equation}
That is, $J^\infty$ is a square matrix with only one non-zero element, $1$, which is at the left, top corner.
This is so because as the power $n$ goes to infinity, every Jordan block of $J$, except the first one,
goes to a zero matrix. Proof: consider such a Jordan block $J_i$, corresponding to the eigenvalue $\lambda_i$,
and suppose it is a $k\times k$ matrix. Therefore,
\[J_i = \left[\begin{array}{ccccc}
   \lambda_i &1   &0  &\ldots &0\\
    0   &\lambda_i &1 &\ldots &0\\
    .   &.         &. &\ldots &.\\
    0   &0   &\ldots &\lambda_i &1\\
    0   &0   &\ldots &0   &\lambda_i
\end{array}\right]\]
Now, the matrix $J_i^n$ can be shown\footnote{by induction, alternatively, noting that $J_i = D_i+U_i$, $D_i$ diagonal
and $U_i$ nilpotent, and noting that these two commute.} to be  
\[J_i = \left[\begin{array}{ccccc}
   \lambda_i^n &\binom(n)(1)\lambda_i^{n-1}   &\binom{n}{2}\lambda_i^{n-2}  &\ldots &\binom{n}{k-1}\lambda_i^{n-k+1}\\
    0   &\lambda_i^n &\binom{n}{1}\lambda_i^{n-1} &\ldots &\binom{n}{k-2}\lambda_i^{n-k+2}\\
    .   &.         &. &\ldots &.\\
    0   &0   &\ldots &\lambda_i^n &\binom{n}{1}\lambda_i^{n-1}\\
    0   &0   &\ldots &0   &\lambda_i^n
\end{array}\right]\] 
That is, the first row of $J_i^n$ is the first $k$ terms of the binomial expansion of $(\lambda_i + 1)^n$, and
the other rows are obtained by right shifting the first row successively and placing zeros initially. Perron
guarantees that each $|\lambda_i| < 1$. Therefore, each non-zero element in $J_i$ goes to $0$ as $n$ goes to infinity,
each such term element being a ratio of a polynomial in $n$ to an exponential in $n$. Thus, we have \ref{j-infinity}.

Let us define $P^\infty$ as 
\[\P^\infty\defeq\lim_{n\rightarrow\infty}P^n\]
We note that $J^{\infty}$ is a rank $1$ matrix, therefore $P^{\infty}$, being equal to $MJ^{\infty}M^{-1}$, is also of
rank $1$. Next, we note that $P^{\infty}$ is a stochastic matrix, as $P$ is stochastic, and the product of two stochastic
matrices is also stochastic. Therefore, being stochastic, every row of $P^{\infty}$ is stochastic, and being of rank $1$,
all the rows are identical. Thus, for some probability distribution on $\Omega$, $\pi = [\pi_1 \pi_2 \cdots \pi_{|\Omega|}]$,
\begin{equation}
\label{p-infinity}
P^{\infty} = \left[\begin{array}{cccc}
      \pi_1  &\pi_2    &\ldots  &\pi_{|\Omega|}\\ 
      \pi_1  &\pi_2    &\ldots  &\pi_{|\Omega|}\\  
        .    &.     &\ldots  &. \\
      \pi_1  &\pi_2    &\ldots  &\pi_{|\Omega|} 
     \end{array}\right]
\end{equation}
Next, we show that $\pi$ as above is a stationary distribution of $P$, i.e., $\pi P = P$.\\
\[P^\infty P \defeq \lim_{n\rightarrow\infty}P^n P = \lim_{n\rightarrow\infty}P^{n+1}\defeq P^\infty\]
We have thus $P^\infty P = P^\infty$. Equating the first row of the LHS matrix product with that of the RHS matrix
we get $\pi P = \pi$, establishing $\pi$ as a stationary distribution of $P$.

Next, we see that $\pi$ is the unique stationary distribution of $P$ and the support of $\pi$ is the entire state space.
Repeating an earlier argument, taking the transpose of $\pi P = P$, we see $\pi^T$ to be\footnote{we note again that we are abusing the standard notation in writing $\pi^T$ as a 
column vector.} a (right) eigenvector of $\pi$ which corresponds to the
eigenvalue $1$ of $P^T$. Since $P^T$ has the same spectra as $P$, and as we have proved $1$ to be the largest eigenvalue
of $P$, we have $1$ to be the largest eigenvalue of $P^T$ as well. From Perron's theorem then, since $1$ is of 
multiplicity $1$, $\pi^T$ is, after normalization, the unique eigenvector corresponding to the eigenvalue $1$. So, $\pi$
is the unique stationary distribution of $P$.
Further, Perron's theorem also guarantees 
that $\pi^T$ is positive. Therefore, the support of $\pi$ is the entire state space. Finally, if we run the chain with any 
initial distribution $\sigma$, as $\sigma P^\infty = \pi$, the chain will converge to the unique stationary distribution $\pi$
in the limit. This completes the present proof of the fundamental theorem.   
\end{proof}

\section{Concluding remarks}
An elementary proof need not be simple, and simplicity in many simple proofs is due to the use of some
advanced concepts. The first proof that we have seen in this note is remarkable because it is both
elementary and simple. Moreover, although the result is about a stochastic process, the proof does not use
any probabilistic idea. The key idea that is used is that when we take the dot product
of a row of a stochastic matrix with a column of a positive matrix, we perform a weighted averaging
of the column entries-- the result will be a value between the smallest and the largest entry
of the column, and the betweeness is strict when the stochastic matrix is positive.
The second proof that we have seen makes essential use of probability arguments. Though the presentation
here (following \cite{mixing-book}) is elementary, the proof originally emanates, as noted in \cite{mitzen},
from renewal theory. The third proof rests on Perron's theorem which is usually proved making use of Gelfand's
spectral radius formula, a result from the theory of Banach algebras, though elementary proofs of 
Perron-Frobenius theorem do exist, see, e.g., \cite{suzu}. One may therefore say that the 
three proofs rest on three different intuitions. The hallmark of a great result is that it can be arrived at 
through different points of view-- indeed then, the fundamental theorem of Markov chains possesses this hallmark.  

{\small{\bf Acknowledgements:} I express my indebtedness to all the authors whose proofs I have surveyed in this paper.
I am grateful to Manindra Agrawal, Rajeeva Karandikar, Satyadev Nandakumar, Nandini Nilakantan, and Nisheeth Vishnoi for 
helpful discussions.}

\end{document}